\documentclass[letterpaper,11pt]{amsart}

\usepackage[utf8]{inputenc}
\usepackage[T1]{fontenc}
\usepackage[english]{babel}

\usepackage[margin=5em]{geometry}

\usepackage{amssymb, amsfonts}
\usepackage{bbm}
\usepackage{enumerate}
\numberwithin{equation}{section}

\newtheorem{theoremcounter}{theoremcounter}[section]

\newtheorem{corollary}[theoremcounter]{Corollary}
\newtheorem{definition}[theoremcounter]{Definition}
\newtheorem{example}[theoremcounter]{Example}
\newtheorem{lemma}[theoremcounter]{Lemma}
\newtheorem{proposition}[theoremcounter]{Proposition}
\newtheorem{remark}[theoremcounter]{Remark}
\newtheorem{remarks}[theoremcounter]{Remarks}
\newtheorem{theorem}[theoremcounter]{Theorem}

\newcommand{\ignore}[1]{\relax}

\newcommand{\nbd}{\nobreakdash-\hspace{0pt}}

\renewcommand{\frak}{\ensuremath{\mathfrak}}
\newcommand{\bboard}{\ensuremath{\mathbb}}

\newcommand{\frakc}{\ensuremath{\frak{c}}}

\newcommand{\frake}{\ensuremath{\frak{e}}}
\newcommand{\frakf}{\ensuremath{\frak{f}}}

\newcommand{\bbM}{\ensuremath{\bboard M}}

\newcommand{\bbP}{\ensuremath{\bboard P}}

\newcommand{\bbS}{\ensuremath{\bboard S}}

\newcommand{\rmM}{\ensuremath{\mathrm{M}}}

\newcommand{\rmS}{\ensuremath{\mathrm{S}}}

\newcommand{\amid}{\ensuremath{\mathop{\mid}}}

\newcommand{\ZZ}{\ensuremath{\mathbb{Z}}}
\newcommand{\QQ}{\ensuremath{\mathbb{Q}}}
\newcommand{\RR}{\ensuremath{\mathbb{R}}}
\newcommand{\CC}{\ensuremath{\mathbb{C}}}

\renewcommand{\Re}{\ensuremath{\mathop{\mathfrak{Re}}}}

\newcommand{\isdiv}{\amid}

\renewcommand{\pmod}[1]{\ensuremath{\;(\mathrm{mod}\; #1)}}

\newcommand{\sgn}{\ensuremath{\mathrm{sgn}}}

\newcommand{\SL}[1]{\ensuremath{\mathrm{SL}_{#1}}}
\newcommand{\Mp}[1]{\ensuremath{\mathrm{Mp}_{#1}}}

\newcommand{\T}{\ensuremath{\mathrm{t}}}
\newcommand{\rT}{\ensuremath{{}^\T\hspace{0.03em}}}

\newcommand{\slashdiv}{\ensuremath{\mathop{/}}}
\newcommand{\HS}{\mathbb{H}}

\newcommand{\td}{\tilde}

\newcommand{\ov}{\overline}

\begin{document}

\title{Holomorphic projections and Ramanujan's mock theta functions}

\author{Özlem Imamo\u glu}
\address{ETH Zürich, Mathematics Dept., CH-8092, Zürich, Switzerland}
\email{ozlem@math.ethz.ch}

\author{Martin Raum}
\address{ETH Zürich, Mathematics Dept., CH-8092, Zürich, Switzerland}
\email{martin.raum@math.ethz.ch}
\urladdr{http://www.raum-brothers.eu/martin/}
%\curraddr[]{}
%\thanks{This paper is in final form and no version of it will be submitted for
%publication elsewhere.}

\author{Olav K. Richter}
\address{Department of Mathematics\\University of North Texas\\ Denton, TX 76203\\USA}
\email{richter@unt.edu}

\thanks{The first author was partially supported by SNF grant 200021-132514.  The second author was partially supported by the ETH Zurich Postdoctoral Fellowship Program and by the Marie Curie Actions for People COFUND Program.  The third author was partially supported by Simons Foundation Grant $\#200765$}

%\dedicatory{Dedicated to Professor XY on the occasion of his seventieth birthday.}

% \date{\today}

\begin{abstract}
We employ spectral methods of automorphic forms to establish a holomorphic projection operator for tensor products of vector-valued harmonic weak Maass forms and vector-valued modular forms.  We apply this operator to discover simple recursions for Fourier series coefficients of Ramanujan's mock theta functions.
\end{abstract}
\maketitle

\section{Introduction}

Mock theta functions have a long history but recent work establishes surprising connections with different areas of mathematics and physics.  
For example, they impact the theory of Donaldson invariants of $\CC\mathbb{P}^2$ that are related to gauge theory (for example, see \cite{Ono-Mal-GeomTop12, Ono-Mal-CNTP12, Ono-Gri-Mal-Forum}), they are intimately linked to the Mathieu and umbral moonshine conjectures (see \cite{Egu-Oog-Tac} and \cite{Ch-Du-Ha}), and they play an important role in the study of quantum black holes and mock modular forms (see~\cite{DMZ}).  For a good overview of mock theta functions, see \cite{Z-Bourbaki} or \cite{Ono-CDM08}.
 
A highlight in the theory of mock theta functions is Zwegers's~\cite{Zwe-thesis, Zwe-Contemp01} ``completion'' of mock theta functions to real-analytic vector-valued modular forms.
That completion of mock theta functions has led to several applications such as the solution of the Andrews-Dragonette conjecture in~\cite{B-O-Invent06}, which provides an explicit formula for the Fourier series coefficients of the third order mock theta function
\begin{gather*}
%\label{eq:definition-of-f(q)}
  f(q)
:=\sum_{n=0}^{\infty}c(f;n)q^n:=
  1 + \sum_{n=1}^{\infty} \frac{q^{n^2}}{(1 + q)^2(1 + q^2)^2 \cdots (1 + q^n)^2}\text{.}
\end{gather*}
The formula for the coefficients $c(f;n)$ in~\cite{B-O-Invent06} is given as an infinite series of Kloosterman sums and $I$-Bessel functions, and closely resembles Rademacher's series representation of the partition function.  In particular, the terms that occur are transcendental.  

In this paper, we determine simple finite recursions for Fourier series coefficients of mock theta functions that depend only on divisor sums, and where all occurring terms are rational.  Our results are in the spirit of Hurwitz's~\cite{Hur-MathAnn1885} class number relations (see also~\cite{Hi-Za-Invent76}):
\begin{gather}
\label{eq:Hurwitz-relations}
 \sum_{\substack{ m\in\ZZ \\
                   m^2\leq 4N}}H(4N-m^2)
=
  2 \sigma(N)
  -
  \sum_{\substack{a,b \in \ZZ \\
                   a,b > 0 \\
                   N=ab }}
           \min(a,b)\text{,}
\end{gather}
where $H(N)$ is the class number of positive definite binary quadratic forms of discriminant $-N$ and $\sigma(n):=\sum_{0 < d \isdiv n} d$.  Specifically, we prove the following relations for the Fourier series coefficient $c(f;\, n)$ of~$f(q)$, where we use the conventions that $ \sigma(n) = 0$, if~$n\not\in\ZZ$, and that $\sgn^+(n) := \sgn(n)$ for $n \ne 0$ and $\sgn^+(0) := 1$.
\begin{theorem}
\label{thm:relations-of-mock-theta-coefficients}
Fix $0 < n \in \ZZ$, and for $a, b \in \ZZ$ set $N := \frac{1}{6}(-3 a + b - 1)$ and ${\td N} := \frac{1}{6} (3 a + b - 1)$.  Then
\begin{gather}
\label{eq:thm:relations-of-mock-theta-coefficients:first-relation}
  \sum_{\substack{ m \in \ZZ \\ 3m^2+m\leq2n}}
  \hspace{-0.9em}
  (m + \tfrac{1}{6})\, c(f; n - \tfrac{3}{2} m^2 - \tfrac{1}{2} m)
\;=\;
  \tfrac{4}{3} \sigma(n) - \tfrac{16}{3} \sigma(\tfrac{n}{2})
  - 2 \sum_{\substack{ a, b \in \ZZ \\2n = a b}}
    \sgn^+(N)\, \sgn^+({\td N})\;
    \big( |N + \tfrac{1}{6}| - |{\td N} + \tfrac{1}{6}| \big)
\text{,}
\end{gather}
where the sum on the right hand side runs over $a, b$ for which $N, {\td N} \in \ZZ$.
\end{theorem}
In Theorem~\ref{thm:relations-of-mock-theta-coefficients-2} we give a similar formula if $n \in \frac{1}{2} \ZZ$ and also relations for the Fourier series coefficients of the mock theta function~$\omega(q)$.  

\begin{example}
Observe that all sums in Theorems~\ref{thm:relations-of-mock-theta-coefficients} and~\ref{thm:relations-of-mock-theta-coefficients-2} are finite, and that only a few terms are needed to find the actual Fourier series coefficients of~$f$.  For example, \eqref{eq:thm:relations-of-mock-theta-coefficients:first-relation} implies that
\begin{multline*}
  -\tfrac{5}{6} c(f;\, 0) + \tfrac{1}{6} c(f;\, 1)
\\
=\;
  \tfrac{4}{3} \sigma(1)
  - \tfrac{16}{3} \sigma( \tfrac{1}{2} )
  - 2 \Big(
    \sgn^+(-1)\, \sgn^+(1) \big(|-1 + \tfrac{1}{6}| - |1 + \tfrac{1}{6}|\big)
    \;+\;
    \sgn^+(0)\, \sgn^+(-1) \big(|0 + \tfrac{1}{6}| - |-1 + \tfrac{1}{6}|\big)
  \Big)
\text{,}
\end{multline*}
showing that $c(f;\, 1) = 1$ (using that $c(f;\, 0) = 1$).
\end{example}
\begin{remarks}
\begin{enumerate}
\item
Jeremy Lovejoy pointed out to us that simple finite recursions for Fourier series coefficients of mock theta functions that depend only on divisor sums can sometimes also be furnished by Appell sums, since these are typically expressible in terms of divisors.  However, it is not clear if  Theorems~\ref{thm:relations-of-mock-theta-coefficients} and~\ref{thm:relations-of-mock-theta-coefficients-2} could be obtained using this idea.

\item
Let $1 < M$ be an odd integer.  Ken Ono indicated to us that Theorem~\ref{thm:relations-of-mock-theta-coefficients} implies that
\begin{gather}
\label{eq:estimates}
  \# \big\{ \text{$n < X$} \,:\, c(f;\, n) \not\equiv 0 \pmod{M} \big\}
\gg
  \frac{\sqrt{X}}{\log X}
\text{.}
\end{gather}
The case $M = 2$, which we have excluded, can be deduced from work of \cite{B-Y-Z-Crelle04}. Scott Ahlgren mentioned to us that it could also be derived from \cite{Ni-Ru-Sa-JNT98}, since $f(q)$ is congruent modulo $2$ to the generating function for the partition function.  For odd $M$, one can apply Theorem~\ref{thm:relations-of-mock-theta-coefficients} as follows:  If $n \ge 7$ is prime, then it is easy to verify that the right hand side of~\eqref{eq:thm:relations-of-mock-theta-coefficients:first-relation} equals $\frac{4}{3} (n + 4)$.  With the help of Dirichlet's prime number theorem one finds that asymptotically \mbox{$(1 - \phi(M)^{-1})\,  X \slashdiv  \log X$} primes $n < X$ give a non-vanishing right hand side of~\eqref{eq:thm:relations-of-mock-theta-coefficients:first-relation}, where $\phi$ is the Euler $\phi$\nbd function.  At most $\sqrt{X}$ of such primes are contained in the progression $n - \tfrac{3}{2} m^2 - \tfrac{1}{2} m$, which yields the desired bound.
\end{enumerate}
\end{remarks}

The proofs of Theorems~\ref{thm:relations-of-mock-theta-coefficients} and~\ref{thm:relations-of-mock-theta-coefficients-2} are based on Zwegers's idea of using holomorphic projection of scalar-valued functions to study mock modular forms (see \cite{And-Rho-Zwe-ANT13} for another application of this idea).  We start by extending the concept of holomorphic projection to tensor products of vector-valued harmonic weak Maass forms of weight $k$ and vector-valued modular forms of weight $l$ (see Theorems~\ref{thm:modular-holomorphic-projection} and~\ref{thm:coefficients-of-holomorphic-projection}), where $k+l\geq 2$.  The case $k+l = 2$ is subtle and features vector-valued quasimodular forms.  Our proof relies on the spectral theory of automorphic forms, and is quite different from the proof of the scalar-valued case in~\cite{G-Z-Invent86}.  We apply our results to the mock theta functions $f(q)$ and $\omega(q)$ (in which case $k=\frac{1}{2}$ and $l=\frac{3}{2}$) to obtain the explicit recursions for their Fourier series coefficients in Theorems~\ref{thm:relations-of-mock-theta-coefficients} and~\ref{thm:relations-of-mock-theta-coefficients-2}.
   
Finally, as already hinted by the similarity of the relations in (\ref{eq:Hurwitz-relations}) and (\ref{eq:thm:relations-of-mock-theta-coefficients:first-relation}), our method also allows one to recover the Hurwitz class number relations in~(\ref{eq:Hurwitz-relations}).  It is conceivable that the method applies to even further classes of automorphic forms, but in this paper we only focus on Ramanujan's mock theta functions.

\vspace{1ex}

\noindent
{\it Acknowledgments}:
We thank Scott Ahlgren, Jeremy Lovejoy, Ren\'{e} Olivetto, and Ken Ono for many helpful comments on an earlier version of this paper.

\section{The metaplectic cover and quasimodular Forms}

We briefly introduce some standard notation needed for the definition of vector-valued automorphic forms.  Let $\zeta_r:=e^{\frac{2\pi i}{r}}$, $\HS := \{ \tau = x + i y \in \CC \,:\, y > 0\}$ be the Poincar\'e upper half plane, and $q := e^{2\pi i\tau}$.  The Fourier series coefficients of a periodic function~$F$ on $\HS$ are always denoted by~$c(F;\, n; y)$.  If this coefficient is constant, then we suppress the dependence on~$y$, and write~$c(F;\, n)$.  Recall that the metaplectic cover $\Mp{2}(\ZZ)$ of $\SL{2}(\ZZ)$ is the group of pairs $(g, \omega)$, where $g = \left(\begin{smallmatrix}a & b \\ c & d\end{smallmatrix}\right)\in \SL{2}(\ZZ)$ and $\omega \,:\, \HS \rightarrow \CC,\, \tau \mapsto \sqrt{c \tau + d}$ for a holomorphic choice of the square root, with group law
$(g_1, \omega_1)(g_2, \omega_2):=(g_1 g_2,\, (\omega_1 \circ g_2) \cdot \omega_2)$.  We usually write $\gamma$ for elements in $\Mp{2}(\ZZ)$.  Standard generators of $\Mp{2}(\ZZ)$ are
\begin{gather*}
  T
:=
  \big( \left(\begin{smallmatrix}1 & 1 \\ 0 & 1\end{smallmatrix}\right),\, 1 \big)
\quad\text{and}\quad
  S
:=
  \big( \left(\begin{smallmatrix}0 & -1 \\ 1 & 0\end{smallmatrix}\right),\, \sqrt{\tau} \big)
\text{,}
\end{gather*}
where $\sqrt{\tau}$ is the principle branch mapping $i$ to $\zeta_8$.

Throughout this paper, $\rho$ denotes a finite dimensional, unitary representation of~$\Mp{2}(\ZZ)$.  If $\pm \rho(S)^2$ is the identity, then $\rho$ factors through~$\SL{2}(\ZZ)$.  Let $V(\rho)$ be the representation space of $\rho$, 
%equipped with a basis. 
and $\langle \cdot\, ,\,\cdot \rangle_\rho$ be the scalar product for which $\rho$ is unitary. 
For fixed half-integer $k$ and for all $\gamma=\left(\left(\begin{smallmatrix}a & b\\c & d\end{smallmatrix}\right), \sqrt{c\tau+d}\right)\in\Mp{2}(\ZZ)$ define the weight~$k$ slash operator of type~$\rho$ on functions $F :\, \HS \rightarrow V(\rho)$:
\begin{gather*}
  \big( F \big|_{k, \rho}\, \gamma\big)(\tau)
:=
  \rho(\gamma)^{-1}\,
   (\sqrt{c\tau+d})^{-2k} F\big( \frac{a \tau + b}{c \tau + d} \big)\text{.}
\end{gather*}
The space $\rmM_{k}(\rho)$ of modular forms of weight~$k$ and type~$\rho$ consists of $|_{k, \rho}$ invariant and holomorphic functions $\HS \rightarrow V(\rho)$ that are bounded at infinity.  Quasimodular forms are important generalizations of modular forms that were introduced in~\cite{Kan-Z}.  A crucial example is the weight $2$ Eisenstein series
\begin{gather} 
\label{eq:E_2}
E_2(\tau):=1 - 24 \sum_{n = 1}^\infty \sigma(n) q^n\text{,}
\end{gather}
whose completion $E_2^*(\tau):=E_2(\tau)-\frac{3}{\pi y}$ is a real-analytic modular form of weight $2$.  We now extend the definition in~\cite{Kan-Z} to the case of vector-valued forms.
\begin{definition}
\label{def:quasimodular}
Let $F :\, \HS \rightarrow V(\rho)$ be a holomorphic function.  Then $F$ is a quasimodular form of weight~$k$ and type~$\rho$, if there is a finite collection~$F_n :\, \HS \rightarrow V(\rho)$ ($0 < n \in \ZZ$) of holomorphic functions such that the following holds:
\begin{enumerate}
\item $(F + \sum_n y^{-n} F_n) \big|_{k, \rho} \, \gamma = (F + \sum_n y^{-n} F_n)$ for all $\gamma \in \Mp{2}(\ZZ)$.
\item $F(\tau) = O(1)$ and $F_n(\tau) = O(1)$ for all $n$ as $y \rightarrow \infty$.
\end{enumerate}
Let $\widetilde{\rmM}_k(\rho)$ be the space of quasimodular forms of weight~$k$ and type~$\rho$. 
\end{definition}
\noindent
 The maximal $n$ with $F_n \ne 0$ in Definition \ref{def:quasimodular} is called the depth of~$F$.  One can show that for this choice of $n$, $F_n$ is a modular form of weight~$k - 2 n$, so that the depth is bounded for fixed~$k$.  We conclude this section with two propositions on quasimodular forms of weight~$2$. 

\begin{proposition}
\label{prop:quasimodular-forms-of-weight-2}
Suppose that $\rho$ is irreducible.  If $\rho$ is the trivial representation, then $\widetilde{\rmM}_2(\rho) = \langle E_2 \rangle$.  Otherwise, $\widetilde{\rmM}_2(\rho) = \rmM_2(\rho)$.
\end{proposition}
\begin{proof}
The first part was proved in~\cite{Kan-Z}.  Suppose that $\rho$ is a not the trivial representation.  Let $F \in \widetilde{\rmM}_2(\rho)$ and consider its completion $F^*(\tau) = F(\tau) + y^{-1} F_2(\tau) \in \rmM_2(\rho)$.  Then $F_2 \in \rmM_0(\rho)$, and $F_2 = 0$, since $\rho$ is non-trivial.  Hence $F^* = F \in \rmM_2(\rho)$.
\end{proof}

For any vector space $V$, we write $\bbP(V): = (V \setminus \{ 0 \}) \slashdiv \CC^\times$ for its projectivization.  We call $w \in \bbP( V(\rho) )$ a cusp of $\rho$, if any lift of $w$ to $V(\rho)$ is a fixed vector of $\rho(T)$.   

\begin{proposition}
\label{prop:eisenstein-series-of-weight-2}
Suppose that $\rho$ is non-trivial and irreducible.  If $\rho\big(\left(\begin{smallmatrix} -1 & 0 \\ 0 & -1 \end{smallmatrix}\right)\big)$ is the identity, then for each cusp~$w \in \bbP( V(\rho) )$ of $\rho$, there is an Eisenstein series~$E_{2; \rho, w} \in \rmM_{2}(\rho)$ with $\langle c( E_{2; \rho, w};\, 0 ),\, w\rangle_{\rho} \ne 0$ and $\langle c( E |_{2; \rho, w};\, 0 ),\, w' \rangle_{\rho} = 0$ for all $w' \in \bbP( V(\rho) )$ with $\langle w,\, w' \rangle_\rho = 0$.
\end{proposition}
\begin{proof}
We employ ``Hecke's trick'' to construct a quasimodular form $E_{2; \rho, w}$ with constant coefficient $w + O(y^{-1})$.  This will yield the desired result, since   $\widetilde{\rmM}_2(\rho) = \rmM_2(\rho)$.  More precisely, set
\begin{gather*}
  E_{2, \epsilon; \rho, w}
:=
  \sum_{\gamma \in \Gamma_\infty \backslash \SL{2}(\ZZ)} \hspace{-0.8em}
  |c \tau + d|^{- 2 \epsilon}\;
  w \big|_{2, \rho}\, \gamma 
\text{,}
\end{gather*}
and $E_{2; \rho, w}:= \lim_{\epsilon \rightarrow 0} E_{2, \epsilon; \rho, w}$.  It is easy to see that the Fourier series expansion of $E_{2, \epsilon, \rho, w}$ is given by
\begin{gather*}
  w
  +
  2 \sum_{\substack{c > 0 \\ d \pmod{c}^\times}} c^{-2}
  \sum_{\alpha \in \ZZ} |\tau + \tfrac{d}{c} + \alpha|^{-2 \epsilon}\, (\tau + \tfrac{d}{c} + \alpha)^{-2} \;
  \rho\big(\left(\begin{smallmatrix} * & * \\ c & d \end{smallmatrix}\right)\big) w
\text{.}
\end{gather*}
Consider the Fourier series expansion of the inner sum over $\alpha$.  As in the classical case, one finds that it converges and decays as $y \rightarrow \infty$.  Thus, its Fourier series expansion contains neither the $M$-Whittaker function nor the function $\tau\rightarrow y^{2 \epsilon}$, and is of the form $w'_c y^{-1 - 2 \epsilon} + O(e^{-\delta y})$ for some $\delta > 0$ and $w'_c \in V(\rho)$.  Performing the limit $\epsilon \rightarrow 0$ shows that
\begin{gather*}
  E_{2; \rho, w}(\tau)
=
  w + w' y^{-1} + O(e^{-2\pi y})
\end{gather*}
for some $w' \in V(\rho)$, and $E_{2; \rho, w} \in \widetilde{\rmM}_2(\rho)$.  This completes the proof.
\end{proof}

\section{Holomophic projections}
\label{sec:holomorphic-projection}

The classical holomorphic projection operator maps continuous functions with certain growth and modular behavior to holomorphic modular forms (for example, see~\cite{Sturm-holproj} and~\cite{G-Z-Invent86}).  In this Section,  we extend the holomorphic projection operator to vector-valued forms.

\begin{definition}
\label{def:holomorphic_projection}
Let $V$ be a (finite dimensional) complex vector space.  Suppose that $F :\, \HS \rightarrow V$ is continuous with Fourier expansion
\begin{gather*}
  F(\tau)
=
  \sum_{m \in \QQ} c(F; m; y) q^m
\text{,}
\end{gather*}
and assume that $F$ satisfies:
\begin{enumerate}[1)]
\item 
  $F(\tau) = c_0 + O(y^{-\epsilon})$ for some $\epsilon > 0$, $c_0 \in V$, and as $y \rightarrow \infty$.
\item 
  $c(F; m; y) = O(y^{2 - k})$ as $y \rightarrow 0$ for all $m > 0$\text{.}
\end{enumerate}
Define
\begin{gather}
\label{eq:def:holomorphic_projection}
  \pi_{\rm hol}(F):=  \pi^{(k)}_{\rm hol}(F)
:=
  c_0 + \sum_{0 < m \in \QQ} c(m)q^m
\text{,}\;
\; \text{with}
\quad
  c(m)
:=
  \frac{(4 \pi m)^{k - 1}}{\Gamma(k - 1)}
  \int_0^\infty c(F; m; y) \, e^{- 4 \pi m y} y^{k - 2} \;dy
\text{,}
\end{gather}
where $\Gamma$ is the $Gamma$-function.
\end{definition}

Exactly as in the scalar-valued case, $\pi_{\rm hol}$ preserves holomorphic functions that satisfy the conditions 1) and 2) in Definition~\ref{def:holomorphic_projection}.
\begin{proposition}
\label{prop:holomorphic_projection_self_reproduction}
Let $V$ be a (finite dimensional) complex vector space.  Suppose that $F :\, \HS \rightarrow V$ is holomorphic with Fourier expansion
\begin{gather*}
  F(\tau)
=
  \sum_{0\le m \in \QQ} c(F; m)\, q^m
\text{.}
\end{gather*}
Then $\pi^{(k)}_{\rm hol}(F) = F$.
\end{proposition}
\begin{proof}
The Fourier series coefficients $c(F; m; y)$ of $F$ in Definition~\ref{def:holomorphic_projection}  are constants in $V$, and the claim follows immediately from the integral representation
\begin{gather*}
  \int_0^\infty e^{-4 \pi m y} y^{k - 2} \;dy
=
  (4 \pi m)^{1 - k} \Gamma(k - 1)
\text{.}
\qedhere
\end{gather*}
\end{proof}
The next theorem on holomorphic projections generalizes~Proposition~5.1 on page~288 and Proposition~6.2 on page~295 of~\cite{G-Z-Invent86} to vector-valued modular forms of weight $k$.  The case $k=2$ is delicate.  The proofs in~\cite{G-Z-Invent86} rely on Poincaré series (and ``Hecke's trick'' if $k=2$), while our proof is based on spectral methods. 
\begin{theorem}
\label{thm:modular-holomorphic-projection}
Fix $2 \le k \in \frac{1}{2}\ZZ$ and a representation $\rho$ of $\Mp{2}(\ZZ)$.  Let $F : \HS \rightarrow V(\rho)$ be a continuous function which satisfies
\begin{enumerate}[(i)]
\item $F \big|_{k, \rho}\, \gamma = F$ for all $\gamma \in \Mp{2}(\ZZ)$.
\item $F(\tau) = c_{0} + O(y^{-\epsilon})$ for some $\epsilon > 0$, $c_0 \in V(\rho)$, and as $y \rightarrow \infty$.
\end{enumerate}
If $k>2$, then $\pi_{\rm hol}(F) \in \rmM_k(\rho)$, and if $k=2$, then $\pi_{\rm hol}(F) \in \widetilde{\rmM}_2(\rho)$.
\end{theorem}
\begin{proof}
By decomposing $\rho$ into a direct sum of irreducible representations, we assume without loss of generality that $\rho$ is irreducible.  Moreover, we may (and do) assume that $c_0 = 0$, i.e., $F = O(y^{-\epsilon})$ as $y \rightarrow \infty$:  If $\rho$ is trivial, then replace $F$ by $F - c_0 E_2^*$.  If $\rho$ is not trivial, then replace $F$ by $F - E_{c_0}$, where $E_{c_0}$ is an Eisenstein series whose constant coefficient equals $c_0$, which exists by Proposition~\ref{prop:eisenstein-series-of-weight-2}.

Conditions $(i)$ and $(ii)$ yield that for every linear functional $\frakf :\, V(\rho) \rightarrow \CC$, the evaluation $\frakf(F)$ belongs to $L^2(\Gamma \backslash \HS)$, where $\Gamma$ is the kernel of $\rho$, and hence a congruence subgroup of $\Mp{2}(\ZZ)$.  Let $\langle \cdot ,\cdot\rangle$ denote the Petersson scalar product, which we extend to vector valued modular forms by applying it componentwise.  We have the following spectral decomposition (for example, see~\cite{Iw-2002}):
\begin{gather}
\label{eq:thm:modular-holomophic-projection:spectral-expansion}
  F
=
  \sum_j \langle g_j, F \rangle\, g_j
  \;+\;
  \sum_j \langle u_j, F \rangle\, u_j
  \;+\;
  \sum_{\frakc}
  \int_{-\infty}^\infty \langle E_{\frakc, k, \frac{1}{2} + ir}, F \rangle\,
                    E_{\frakc, k, \frac{1}{2} + ir} \; dr
\text{,}
\end{gather}
where $\{g_j\}$ is a complete orthonormal system of holomorphic modular forms, $\{u_j\}$ is a complete orthonormal system of proper Maass cusp forms and residual contributions, $E_{\frakc, k, \frac{1}{2} + ir}$ is the Eisenstein series for the cusp~$\frakc$ of weight~$k$ with spectral parameter $\frac{1}{2} + ir$, and the last sum runs over cusps~$\frakc$ of $\Gamma \backslash \HS$.  We show that holomorphic projection simply picks the holomorphic components in the spectral expansion:
\begin{gather}
\label{eq:holomorphic-projection-picks-holomorphic-component}
  \pi_{\rm hol}( F )
=
  \sum_j \langle g_j, F \rangle\, g_j
\text{.}
\end{gather}
This will prove the theorem, since modular transformations preserve each of the three sums in \eqref{eq:thm:modular-holomophic-projection:spectral-expansion}.

The spectral expansion converges pointwise absolutely and uniformly on compact sets, and we find that
\begin{gather*}
  \pi_{\rm hol}(F)
=
  \sum_j \langle g_j, F \rangle\, g_j
  \;+\;
  \sum_j \langle u_j, F \rangle\, \pi_{\rm hol}(u_j)
  \;+\;
  \sum_{\frakc}
  \int_{-\infty}^\infty \langle E_{\frakc, k, \frac{1}{2} + ir}, F \rangle\,
                    \pi_{\rm hol}(E_{\frakc, k, \frac{1}{2} + ir}) \; dr
\text{.}
\end{gather*}
Write $\lambda = s (1 - s) = (s - \frac{k}{2})(1 - \frac{k}{2} - s) + \frac{k}{2}(1 - \frac{k}{2})$ for the eigenvalues under the weight~$k$ Laplace operator
\begin{gather}
\label{eq:definition-of-laplace-operator}
  \Delta_k
:=
  - \xi_{2-k} \circ \xi_k
=
  - 4 y^2 \frac{\partial}{\partial \tau} \frac{\partial}{\partial \overline \tau}
  + 2 k i y\frac{\partial}{\partial \overline \tau}
\end{gather}
(see page~29 of~\cite{B-LNM}), where the operator $\xi_k:= 2 i y^k \ov{\frac{\partial}{\partial \overline{\tau}}}$ was introduced in \cite{B-F-Duke04}.  Since the operator $\Delta_k$ is non-negative, we have either $s = \frac{1}{2} + i r$ ($r \in \RR$) or $0 \le s \le 1$.  The latter case does not occur for the third sum in~\eqref{eq:thm:modular-holomophic-projection:spectral-expansion}, and for the second, it actually is $0 < s < 1$:  If  $s \in \{0, 1\}$, then $\lambda = 0$, and Maass cusp forms of weight $k \geq 2$ and eigenvalue $0$ are holomorphic, since they are in the kernel of $\xi_k$, whose image consists of holomorphic cusp form of weight~$2 - k$.  Fix a Maass form $u$ of weight $k$ and eigenvalue $s(1-s)$.  Then
\begin{gather*}
  u(\tau)
=
  \sum_{m \in \QQ} c(u;\, m; y)\, q^m
\text{,}
\qquad
  c(u; m; y)
= 
  c(u;\, m)\,
  y^{-\frac{k}{2}} e^{2 \pi m y} W_{\frac{k}{2}, s - \frac{1}{2}} (4 \pi m y)
\quad
  \text{for $m > 0$}
\text{,}
\end{gather*}
where $c(u;\, m)$ is a constant and $W_{\nu,\mu}$ stands for the usual Whittaker-$W$ function. 
We apply Definition \ref{def:holomorphic_projection} to find that
\begin{gather}
\label{eq:holomorphic-projection-integral-in-modularity-proof}
  c(\pi_{\rm hol}(u); m)
  = \frac{(4 \pi m)^{k - 1}}{\Gamma(k - 1)}\int_0^\infty y^{\frac{k}{2} - 2} W_{\frac{k}{2}, s - \frac{1}{2}} (4 \pi m y) e^{-2 \pi m y} \;dy
\qquad (m>0)\text{.}
\end{gather}

If $k \ge 2$ and $s \not\in \{0, 1\}$, then $\Re(s) - 1 + \frac{k}{2} > 0$ and $\frac{k}{2} - \Re(s) > 0$, and~\eqref{eq:holomorphic-projection-integral-in-modularity-proof} vanishes due to (7.621.11) of~\cite{GR-tables}:
\begin{gather*}
  \int_0^\infty e^{-\frac{1}{2} x} x^{\nu - 1} W_{\kappa, \mu}(x) \; dx
=
  \frac{\Gamma(\nu + \tfrac{1}{2} - \mu) \Gamma(\nu + \tfrac{1}{2} + \mu)}
       {\Gamma(\nu - \kappa + 1)}
\qquad
  \big[ \Re(\nu + \tfrac{1}{2} \pm \mu) > 0 \big]
\text{.}
\end{gather*}
Hence $\pi_{\rm hol}(u_j) = \pi_{\rm hol}(E_{\frakc, k, \frac{1}{2} + i r}) = 0$, which implies \eqref{eq:holomorphic-projection-picks-holomorphic-component}.
\end{proof}

Next we recall the definition of (harmonic) weak Maass forms from \cite{B-F-Duke04}, which involves the weight~$k$ Laplace operator given in~\eqref{eq:definition-of-laplace-operator}.

\begin{definition}
\label{def:harmonic-weak-maass-forms}
Let $k \in \frac{1}{2}\ZZ$ and let $\rho$ be a unitary, finite dimensional representation of $\Mp{2}(\ZZ)$.  A smooth function $F :\, \HS \rightarrow V(\rho)$ is called a harmonic weak Maass form of weight~$k$ and type~$\rho$ if
\begin{enumerate}[(1)]
\item 
    $F |_{k, \rho}\, \gamma = F$ for all $\gamma \in \Mp{2}(\ZZ)$.
\item
$\Delta_k F = 0$.
\item
  $F(\tau) = O(e^{a y})$ as $y \rightarrow \infty$ for some $a > 0$.
\end{enumerate}
Let $\bbM_k(\rho)$ denote the space of harmonic Maass forms of weight $k$ and type~$\rho$, and denote its subspace of weakly holomorphic modular forms by ${\rm M}_{k}^!(\rho) \subset \bbM_k(\rho)$.
\end{definition}

Recall that Proposition 3.2 of \cite{B-F-Duke04} asserts that $\xi_k :\, \bbM_k(\rho) \rightarrow \rmM^!_{2 - k}(\ov{\rho})$.  The space of forms $F \in \bbM_k(\rho)$ with $\xi_k(F) \in S_{2 - k}(\ov{\rho})$ (the space of cusp forms of weight~$2-k$ and type~$\ov{\rho}$) is denoted by~$\bbS_k(\rho)$.  If $F \in \bbS_k(\rho)$, then we write $F = F^+ + F^-$, where
\begin{gather}
\label{eq:harmonic-weak-maass-forms-fourier-expansion}
  F^+(\tau)
:=
  \sum_{-\infty \ll m} c^+(F; m)\, q^m
\text{,}
\qquad
  F^-(\tau)
:=
  - (4 \pi)^{k - 1} \sum_{m < 0} c^-(F; m) |m|^{k - 1}\,\Gamma(1 - k, 4 \pi |m| y)\, q^{m}
\text{,}
\end{gather}
and $\Gamma(s,y):=\int_{y}^{\infty}e^{-t}t^{s-1}\ dt$ is the incomplete Gamma-function.  A straightforward computation shows that $\xi_k(F) = \sum_{0 < m} \ov{c^-(F; -m)}\, q^m$.  The non-holomorphic Eichler integral provides a partial inverse to $\xi_k$.  More precisely, if $G \in \rmS_{2 - k}(\rho)$, then $\xi_{k}( G^* ) = G$, where
\begin{gather*}
  G^*(\tau)
:=
  - (2 i)^{k - 1}
  \int_{- \ov{\tau}}^{i \infty} \frac{\ov{G(-\ov{w})}}{(w + \tau)^{k}} \; dw
=
  - (4 \pi)^{k - 1}
  \sum_{m<0} \ov{c(G; |m|)} |m|^{k - 1}\, \Gamma(1 - k, 4 \pi |m| y)\, q^{m}
\text{.}
\end{gather*}
In particular, if $F \in \bbS_k(\rho)$ and $F^- := \xi_k(f)^*$, then $F^+:= F - F^-$ is holomorphic.

With an abuse of notation, we often write $F G$ instead of $F \otimes G$ for the tensor product of~$F$ and~$G$.
Finally, we give the Fourier series coefficients of $\pi_{\rm hol}(F^- G)$, which feature the hypergeometric series ${}_2 {\rm F}_1 (a,b,c; z) := \sum_{n=0}^{\infty} \frac{(a)_n (b)_n}{(c)_n} \frac{z^n}{n!}$, where $(p)_n:=p(p+1)(p+2)\cdots(p+n-1)$ is the Pochhammer symbol.

\begin{theorem}
\label{thm:coefficients-of-holomorphic-projection}
Let $F \in \bbS_{k}(\rho)$ and $G \in \rmM_{l}(\sigma)$ with $k + l \ge 2$, $k \not=1$.  If $n>0$, then
\begin{gather}
\label{eq:thm:coefficients-formula-of-holomorphic-projection}
  c(\pi_{\rm hol}(F^- G);\, n)
=
  \frac{-(4 \pi)^{k - 1} \Gamma(l)}
       {\Gamma(k + l)}
  n^{k - 1}
  \sum_{\substack{ m + {\td m} = n \\ m < 0 }}
  c^-(F; m) c(G; {\td m}) \,
  \big(\frac{n}{\td m}\big)^l\, {}_2F_1\big(1, l, k + l, \frac{n}{{\td m}} \big)
\text{.}
\end{gather}
\end{theorem}
\begin{proof}
Let $G(\tau)=\sum_{{\td m} \ge 0} c(G; {\td m})\, q^{\td m}$ and $ F^-(\tau) = -(4 \pi)^{k - 1} \sum_{m < 0} c^-(F; m) |m|^{k - 1} \, \Gamma(1 - k, 4 \pi |m| y)\, q^m$ as in~\eqref{eq:harmonic-weak-maass-forms-fourier-expansion}.   We find that
\begin{gather*}
  c(F^- G; n; y)
=
  -(4 \pi)^{k - 1}
  \sum_{\substack{{\td m} + m = n \\ m < 0}}
  |m|^{k - 1} c^-(F; m)\, c(G; {\td m}) \, \Gamma( 1 - k, 4 \pi |m| y )\, q^n
\text{}
\end{gather*}
converges absolutely, since $c(F^-; m)$ and $c(G; {\td m})$ are of polynomial growth, and since
\begin{gather*}
  \Gamma( 1 - k, 4 \pi |m| y )
\asymp
  ( 4 \pi |m| y )^{-k} e^{ -4 \pi |m| y }
\quad
  \text{as }y \rightarrow \infty
\text{.}
\end{gather*}
If $n>0$, then according to Definition \ref{def:holomorphic_projection} we obtain the following expression for~$c(\pi_{\rm hol}(F^- G); n)$:
\begin{gather*}
  -\frac{(4 \pi n)^{k + l - 1} (4 \pi)^{k - 1}}{\Gamma(k + l - 1)}
  \int_0^\infty \sum_{\substack{{\td m} + m = n \\ m < 0}} \hspace{-0.6em}
  |m|^{k - 1} c^-(F; m)\, c(G; {\td m}) \,
  \Gamma( 1 - k, 4 \pi |m| y )\,
  e^{-4 \pi n y} y^{k + l - 2} \;dy
\text{,}
\end{gather*}
where the integral converges, since $k + l \ge 2$ by assumption. A standard argument justifies the interchange of integration and summation, and \eqref{eq:thm:coefficients-formula-of-holomorphic-projection} follows from~(6.455) on page~657 of~\cite{GR-tables} (observing that $l>0$ and $k+l-1>0$), which shows that
\begin{gather*}
  \int_0^\infty
  \Gamma( 1 - k, 4 \pi |m| y )\,
  e^{-4 \pi n y} y^{k + l - 2} \;dy
=
  \frac{(4 \pi |m|)^{1 - k} \Gamma(l)}
       {(k + l - 1) (4 \pi {\td m})^l}\,
       {}_2F_1\big( 1, l, k + l;\, \frac{n}{{\td m}} \big)
\text{.}
\qedhere
\end{gather*}
\end{proof}

\section{Ramanujan's mock theta functions~$f$ and~$\omega$}
\label{sec:relations}

Zwegers suggested holomorphic projection of scalar-valued functions as a tool to investigate mock modular forms, and he applied this idea in his recent joint work~\cite{And-Rho-Zwe-ANT13}.  In this Section, we extend Zwegers's suggestion to vector-valued forms.  More specifically, we apply holomorphic projection to the (tensor) product of a harmonic weak Maass form~$F=F^++F^-$ and a modular form~$G$.  If the holomorphic projections converge, then
\begin{gather}
\label{eq:idea-of-holomorphic-projection}
  \pi_{\rm hol}(F G)
=
  \pi_{\rm hol}(F^- G)
  +
  \pi_{\rm hol}(F^+ G)
=
  \pi_{\rm hol}(F^- G)
  +
  F^+ G
\text{.}
\end{gather}
The left hand side of \eqref{eq:idea-of-holomorphic-projection} is modular or quasimodular by Theorem~\ref{thm:modular-holomorphic-projection}, and the right hand side can be described by Theorem~\ref{thm:coefficients-of-holomorphic-projection}.  If the left hand side can be identified, say in terms of Eisenstein series, then \eqref{eq:idea-of-holomorphic-projection} yields relations for the coefficients of~$F^+$.  We apply this idea to find relations for the Fourier coefficients of the mock theta functions $f(q)$ and
\begin{gather*}
%\label{eq:definition-of-w(q)} 
 \omega(q)
:=\sum_{n=0}^{\infty}c(\omega;n)q^n:=
 \sum_{n=0}^{\infty} \frac{q^{2n^2 + 2n}}{(1 - q)^2(1 - q^3)^2 \cdots (1 - q^{2n+1})^2}
\text{,}
\end{gather*}
 which will prove Theorems~\ref{thm:relations-of-mock-theta-coefficients} and~\ref{thm:relations-of-mock-theta-coefficients-2}. 

First we recall Zwegers's~\cite{Zwe-Contemp01} completion of $f$ and $\omega$.  As before, $q:= e^{2\pi i\tau}$. Set
\begin{gather*}
  F^+(\tau)
:=
  \rT
  \big( q^{\frac{-1}{24}} f(q),\,
        2 q^{\frac{1}{3}} \omega(q^{\frac{1}{2}}),\
        2 q^{\frac{1}{3}} \omega(-q^{\frac{1}{2}})
  \big)
\text{,}
\end{gather*}
and $F^-(\tau):= -2 \sqrt{6} G^*(\tau)$ with
\begin{align}
\label{eq:definition-of-G} 
G(\tau):=  \rT
  \Big(
    - \sum_{n \in \ZZ} (n + \tfrac{1}{6})
                     e^{3 \pi i (n + \tfrac{1}{6})^2\, \tau},\;
    \sum_{n \in \ZZ} (-1)^n (n + \tfrac{1}{3})
                   e^{3 \pi i (n + \tfrac{1}{3})^2\, \tau},\;
    - \sum_{n \in \ZZ} (n + \tfrac{1}{3})
                     e^{3 \pi i (n + \tfrac{1}{3})^2\, \tau}
  \Big)
\text{.}
\end{align}
Theorem~3.6 of~\cite{Zwe-Contemp01} implies that
\begin{gather}
\label{eq:definition-of-F} 
F:= F^+ + F^-\in\bbS_{\frac{1}{2}}(\rho_3)
\text{,}
\end{gather}
where $\rho_3$ is determined by
\begin{alignat*}{2}
  \rho_3(T)
&:=
  \begin{pmatrix}
    \zeta_{24}^{-1} & 0 & 0 \\
    0 & 0 & \zeta_3 \\
    0 & \zeta_3 & 0 \\
  \end{pmatrix}
\text{,}
\qquad
&
  \rho_3(S)
&:=
  \zeta_8^{-1}
  \begin{pmatrix}
    0 & 1 & 0 \\
    1 & 0 & 0 \\
    0 & 0 & -1
  \end{pmatrix}
\text{.}
\end{alignat*}
Moreover, \cite{Zwe-Contemp01} gives the transformations laws of $G$ showing that $G\in S_{\frac{3}{2}}(\ov{\rho_3})$.  We now explore \eqref{eq:idea-of-holomorphic-projection} with $F$ in \eqref{eq:definition-of-F} and $G$ in \eqref{eq:definition-of-G}.  We begin with the left hand side.

\begin{proposition}
\label{prop:holomorphic-projection-of-FG-is-quasimodular}
Let $F$ and $G$ be as in \eqref{eq:definition-of-F} and \eqref{eq:definition-of-G}.  Then
\begin{gather*}
 \pi_{\rm hol}(F G)
\in
  \widetilde{M}_2(\rho_3 \otimes \ov{\rho_3})
\text{.}
\end{gather*}
\end{proposition}
\begin{proof}
Note that $FG$ satisfies the hypotheses of Theorem~\ref{thm:modular-holomorphic-projection} with $k=2$, and hence $\pi_{\rm hol}( FG )\in\widetilde{\rmM}_2(\rho_3 \otimes \ov{\rho_3})$.
\end{proof}

The following two lemmas provide the necessary tools to determine $ \pi_{\rm hol}(F G)$ more explicitly.  First, we decompose the tensor product~$\rho_3 \otimes \ov{\rho_3}$ into irreducible components.  Second, we determine the corresponding spaces of quasimodular forms.  Finally, we express $\pi_{\rm hol}(F G)$ as a concrete quasimodular Eisenstein series.

\begin{lemma}
\label{la:rho3-rho3-decomposition}
The representation $\rho_3 \otimes \ov{\rho_3}$ is isomorphic to $\sigma_1 \oplus \sigma_2 \oplus \sigma_6$, where $\sigma_1$, $\sigma_2$, and $\sigma_6$ are irreducible subrepresentations whose representation spaces are spanned by the columns of the matrices
\begin{gather*}
  \begin{pmatrix}
  1 \\ 0 \\ 0 \\ 0 \\ 1 \\ 0 \\ 0 \\ 0 \\ 1
  \end{pmatrix}
\text{,}\quad
  \begin{pmatrix}
  0 & 1 \\ 0 & 0 \\ 0 & 0 \\ 0 & 0 \\ 1 & -\frac{1}{2} \\ 0 & 0 \\ 0 & 0 \\ 0 & 0 \\ -1 & -\frac{1}{2}
  \end{pmatrix}
\text{,}\quad
  \begin{pmatrix}
  0 & 0 & 0 & 0 & 0 & 0 \\
  1 & 1 & 0 & 0 & 0 & 0 \\
  -1 & 1 & 0 & 0 & 0 & 0 \\
  0 & 0 & 1 & 1 & 0 & 0 \\
  0 & 0 & 0 & 0 & 0 & 0 \\
  0 & 0 & 0 & 0 & 1 & 1 \\
  0 & 0 & 1 & -1 & 0 & 0 \\
  0 & 0 & 0 & 0 & -1 & 1 \\
  0 & 0 & 0 & 0 & 0 & 0
  \end{pmatrix}
\text{,}\quad
  \text{respectively.}
\end{gather*}
\end{lemma}

\begin{proof}
The claim follows easily after forming the Kronecker products of the representation matrices for~$T$ and~$S$.
\end{proof}

Let $E_2$ as in \eqref{eq:E_2}, 
\begin{align*}{2}
  E_2^{[2]}(\tau)
&:=
  \tfrac{1}{12}
  \big( 2 E_2(2\tau) - E_2(\tau) \big)
=
  \tfrac{1}{12}
  \big( 1 + 24 \sum_{0 < n} \sigma(n) (q^n - 2 q^{2 n}) \big)
\text{,}
\quad
\text{and}
\\
%(for example, see A.3 in~\cite{Ch-Du-Ha})
  E_{\sigma_2}(\tau)
&:=
  \rT 
  \big( 6 \big( E_2^{[2]} (\tau)
                - E_2^{[2]} (\tfrac{\tau}{2}) \big),\;
        12 E_2^{[2]}(\tau)
  \big)
\text{.}
\end{align*}

\begin{lemma}
\label{la:explicit-eisenstein-series}
We have $\widetilde{\rmM}_2(\sigma_1) = \langle E_2 \rangle$ and $\rmM_2(\sigma_2) = \langle E_{\sigma_2} \rangle$.  The space $\rmM_2(\sigma_6)$ has dimension~$1$ and is spanned by an Eisenstein series.
\end{lemma}
\begin{proof}

Note that $\sigma_1$ is the trivial representation, and $\rmM_2(\sigma_1) = \{ 0 \}$ and $\widetilde{\rmM}_2(\sigma_1) = \langle E_2 \rangle$ by Proposition \ref{prop:quasimodular-forms-of-weight-2}. 

We next find the dimensions of $\rmM_2(\sigma_2)$ and $\rmM_2(\sigma_6)$.  Observe that $\sigma_2$ and $\sigma_6$ are unitary as subrepresentations of the unitary representation $\rho_3 \otimes \ov{\rho_3}$ (however, the bases given in Lemma~\ref{la:rho3-rho3-decomposition} do not form an orthonormal basis).  The dimension formula on page~228 of~\cite{Borcherds-Gross-Kohnen-Zagier} may be extended to the weight $2$ case to apply to $\sigma_2$ and $\sigma_6$ (see Theorem~6 of~\cite{Sko-Weil}).  For a matrix~$M$ which is diagonalizable over a cyclotomic field, 
set $\alpha(M) := \sum_{i} b_i$, where $e^{2 \pi i\, b_i}$ ($0 \le b_i < 1$) are the eigenvalues of~$M$.  
If $\rho$ is an unitary, finite dimensional representation of~$\Mp{2}(\ZZ)$ with $\rho(S)^2$ the identity,
the dimension formula on page~228 of~\cite{Borcherds-Gross-Kohnen-Zagier} for weight~$2$ states that
\begin{gather*}
  \dim \rmM_2(\rho)
=
  d + \frac{2d}{12} - \alpha(-\rho(S))
  - \alpha(\zeta_3^{-1} \rho(ST)^{-1}) - \alpha(\rho(T))
\text{.}
\end{gather*}
We have
\begin{alignat*}{3}
  \alpha(\sigma_2(T))
&=
%  0 + \frac{1}{2}
  \tfrac{1}{2}
\text{,}
\qquad
&
  \alpha(-\sigma_2(S))
&=
%  0 + \frac{1}{2}
  \tfrac{1}{2}
\text{,}
\qquad
&
  \alpha(\zeta_3^{-1} \sigma_2(S T)^{-1} )
&=
%  0 + \frac{1}{3}
  \tfrac{1}{3}
\text{,}
\quad
\text{and}
\\
  \alpha(\sigma_6(T))
&=
%  0 + \frac{1}{2} + \frac{1}{8} + \frac{5}{8} + \frac{3}{8} + \frac{7}{8}
%=
  \tfrac{5}{2}
\text{,}
&
  \alpha(-\sigma_6(S))
&=
%  0 + 0 + 0 + \frac{1}{2} + \frac{1}{2} + \frac{1}{2}
%=
  \tfrac{3}{2}
\text{,}
&
  \alpha(\zeta_3^{-1} \sigma_6(S T)^{-1} )
&=
%  0 + 0 + \frac{1}{3} + \frac{1}{3} + \frac{2}{3} + \frac{2}{3}
%=
  2
\text{,}
\end{alignat*}
which shows that
\begin{gather*}
  \dim \rmM_2(\sigma_2)
=
  1
\quad \text{and}
\quad
  \dim \rmM_2(\sigma_6)
=
  1
\text{.}
\end{gather*}

Now, $\sigma_2(S)=\left(\begin{smallmatrix}\frac{1}{2} & \frac{3}{4} \\ 1 & -\frac{1}{2}\end{smallmatrix}\right)$ with respect to the basis given in Lemma~\ref{la:rho3-rho3-decomposition}, and $E_2^{[2]}(\frac{-1}{\tau})=-\frac{1}{2}\tau^2E_2^{[2]}(\frac{\tau}{2})$, which yields that $E_{\sigma_2}\big|_{2,\sigma}S=E_{\sigma_2}$ and $E_{\sigma_2} \in \rmM_2(\sigma_2)$.  Hence $\rmM_2(\sigma_2) = \langle E_{\sigma_2} \rangle$. Finally, it is easy to see that $\sigma_6\left(\begin{smallmatrix}-1 & 0 \\ 0 & -1\end{smallmatrix}\right)$ is the identity and Proposition~\ref{prop:eisenstein-series-of-weight-2} implies that $\rmM_2(\sigma_6)$ is spanned by an Eisenstein series.
\end{proof}

 Lemma~\ref{la:rho3-rho3-decomposition} and  Lemma~\ref{la:explicit-eisenstein-series} allow us to write $\pi_{\rm hol}(F G)$ as a specific quasimodular form.

\begin{corollary}
\label{cor:mock-theta-identification-as-modular-form}
We have
\begin{gather}
\label{eq:cor:mock-theta-identification-as-modular-form}
  \pi_{\rm hol}(F G)
=
  -\tfrac{1}{6} E
\text{,}
\end{gather}
where
\begin{alignat*}{6}
  E(\tau)
:= 
  \frake_1
  \big(   \tfrac{1}{3} E_2(\tau)
        + 8 E_2^{[2]}(\tau) \big)
  \;+\;
  \frake_5
  \big(   \tfrac{1}{3} E_2(\tau)
        - 4 E_2^{[2]} (\tfrac{\tau}{2}) \big)
  \;+\;
  \frake_9
  \big(
          \tfrac{1}{3} E_2(\tau)
        - 8 E_2^{[2]}(\tau)
        + 4 E_2^{[2]} (\tfrac{\tau}{2}) \big)
\text{,}
\end{alignat*}
and where $\frake_1,\ldots,\frake_9$ stands for the standard basis of $\RR^9$.
\end{corollary}
\begin{proof}
We apply Lemma~\ref{la:rho3-rho3-decomposition} and  Lemma~\ref{la:explicit-eisenstein-series} to find that $E$ is the unique quasimodular form in $\widetilde{\rmM}_2(\rho_3 \otimes \ov{\rho_3})$ with constant Fourier series coefficient $\frake_1$.  Furthermore, the constant Fourier series coefficient of $\pi_{\rm hol}(F^- G)$ vanishes ($F^- G$ decays rapidly towards infinity) and the constant Fourier series coefficient of $F^+ G$ equals $-\frac{1}{6}\frake_1$.  Thus, the constant Fourier series coefficients of $\pi_{\rm hol}(F G)$ and $-\frac{1}{6}E$ coincide, and the claim follows from Proposition~\ref{prop:holomorphic-projection-of-FG-is-quasimodular} and the uniqueness of $E$.
\end{proof}

Corollary \ref{cor:mock-theta-identification-as-modular-form} permits us to restate \eqref{eq:idea-of-holomorphic-projection} as follows:

\begin{gather}
\label{eq:restate-of-hol(FG)}
- F^+ G
=
  \pi_{\rm hol} (F^- G) + \tfrac{1}{6} E
\text{,}
\end{gather}
where $\pi_{\rm hol} (F^- G)$ is determined by Theorem \ref{thm:coefficients-of-holomorphic-projection}.  Comparing the Fourier series expansions of both sides of \eqref{eq:restate-of-hol(FG)} yields explicit relations for the components of $F^+$, i.e., for the mock theta functions $f$ and $\omega$.  We now demonstrate this to prove Theorem~\ref{thm:relations-of-mock-theta-coefficients}.

\begin{proof} [Proof of Theorem~\ref{thm:relations-of-mock-theta-coefficients}]
Write $G=\rT(G_0,G_1,G_2)$ and $G^*=\rT(G_0^*,G_1^*,G_2^*)$.  Consider the first component of~\eqref{eq:restate-of-hol(FG)}:
\begin{gather*}
%\label{eq:restate-of-hol(FG)-first-component}
  - q^{-\frac{1}{24}} f(q)\, G_0(q)
=
  -2\sqrt{6}\,\pi_{\rm hol} (G_0^*G_0) + \tfrac{1}{18} E_2(\tau)+ \tfrac{4}{3} E_2^{[2]}(\tau)
\text{,}
\end{gather*}
and Theorem~\ref{thm:relations-of-mock-theta-coefficients} follows immediately after inserting the Fourier series expansions of $f$, $G_0$, $E_2$, $E_2^{[2]}$, and $\pi_{\rm hol} (G_0^*G_0)$, which is given by the following Lemma.
\end{proof}

\begin{lemma}
\label{la:mock-theta-holomorphic-projection-coefficients}
Let $a, b \in \ZZ$, and set $N:=\frac{1}{6}(-3 a + b - 1)$ and ${\td N}:=\frac{1}{6} (3 a + b - 1)$.  We have
\begin{gather*}
  c(\pi_{\rm hol}(G^*_0 G_0);\, n)
=
  \frac{1}{\sqrt{6}}
  \sum_{\substack{ a, b \in \ZZ \\2n = a b}}
  \sgn\big( (N + \tfrac{1}{6}) ({\td N} + \tfrac{1}{6}) \big) \,
  \big( |N + \tfrac{1}{6}| - |{\td N}+ \tfrac{1}{6}| \big)
\text{,}
\end{gather*}
where the sum runs over $a,b$ for which $N, {\td N} \in \ZZ$.
\end{lemma}
\begin{proof}
We apply Theorem~\ref{thm:coefficients-of-holomorphic-projection} with $k = \frac{1}{2}$ and $l = \frac{3}{2}$ to find that
\begin{alignat*}{2}
  c(\pi_{\rm hol}(G^*_0 G_0);\, n)
&=
&
  \frac{-1}{4 \sqrt{n}}
&
  \sum_{\substack{ m + {\td m}=n \\ m<0}} \hspace{-0.3em}
  c(G_0;\, |m|) c(G_0;\, {\td m}) \,
  \big(\frac{n}{\td m}\big)^{\frac{3}{2}} \,
  {}_2F_1\big( 1, \tfrac{3}{2}, 2; \frac{n}{{\td m}} \big)\\
&=
&
  \frac{-1}{2}
&
  \sum_{\substack{m + {\td m}=n \\m<0}} \hspace{-0.3em}
  c(G_0;\, |m|) c(G_0;\, {\td m}) \,
  \frac{\sqrt{{\td m}} - \sqrt{|m|}}{\sqrt{{\td m} |m|}}
\text{,}
\end{alignat*}
where the second equality follows from the hypergeometric series identity~(15.4.18) of~\cite{NIST}.
%
%% EXPRESSION THAT NIST GIVES FOR THE HYPERGEOMETRIC FUNCTIONS
% \begin{gather*}
%   \frac{1}{\sqrt{ 1 - \frac{n}{{\td m}} }}
%   \big( \frac{1}{2} + \frac{1}{2} \sqrt{1 - \frac{n}{{\td m}}} \big)^{-1}
% \end{gather*}
The theta series $G_0$ is supported on $-m = \frac{3}{2}( N + \frac{1}{6} )^2$ and ${\td m} = \frac{3}{2}( {\td N} + \frac{1}{6} )^2$ with $N, {\td N} \in \ZZ$.  Thus,
\begin{gather*}
  c(\pi_{\rm hol}(G^*_0 G_0);\, n)
=
  \frac{-1}{\sqrt{6}}
  \sum_{\substack{ N, {\td N} \in \ZZ \\ 2n = ({\td N} - N)\, (3 ({\td N} + N) + 1)}} \hspace{-1.5em}
 (N + \tfrac{1}{6}) ({\td N} + \tfrac{1}{6})\,\frac{|{\td N} + \tfrac{1}{6}| - |N+ \tfrac{1}{6}|}
     {|{\td N}+ \tfrac{1}{6}| |N + \tfrac{1}{6}|}
\text{,}
\end{gather*}
and we obtain the desired result after setting $a: = {\td N} - N$ and $b: = 3({\td N} + N) + 1$.
\end{proof}

Considering different components of~\eqref{eq:restate-of-hol(FG)} yields the relations in the following Theorem, where $\sigma(n) = c(f;\, n) = 0$, if~$n\not\in\ZZ$, and $c_{h}(\omega;\, n) := c(\omega;\, n)$, if $n \equiv h \pmod{2}$, and $0$, otherwise, and as before $\sgn^+(n) := \sgn(n)$ for $n \ne 0$ and $\sgn^+(0) := 1$.

\begin{theorem}
\label{thm:relations-of-mock-theta-coefficients-2}
Fix $n \in \frac{1}{2} \ZZ$, and for $a,b\in\ZZ$ with $8 n + 1 = a b$ set $N := \frac{1}{12} (3 a - b - 2)$ and ${\td N} := \frac{1}{12} ( 3 a + b - 4 )$.  Then
\begin{gather}
\label{eq:thm:relations-of-mock-theta-coefficients:second-relation}
  \sum_{\substack{ m \in \ZZ \\3m^2+2m\leq2n}} \hspace{-0.9em}
  (m + \tfrac{1}{3})\,
  c(f;\, n - \tfrac{3}{2} m^2 - m )
\;=\;
  - 2
  \sum_{\substack{a, b \in \ZZ \\ 8n + 1 = a b}} \hspace{-0.2em}
  \sgn^+(N)\, \sgn^+({\td N})\;
  \big( |N + \tfrac{1}{6}| - |{\td N} + \tfrac{1}{3}| \big)
\text{,}
\end{gather}
where the sum on the right hand side runs over $a, b$ for which $N, {\td N} \in \ZZ$.

Fix $h \in \{0, 1\}$ and $n \in \ZZ + \frac{h}{2}$.  For $a,b\in\ZZ$ with $ 8 n + 3 = a b$ set $N := \frac{1}{12} ( 3 a - b - 4 )$ and ${\td N} := \frac{1}{12} ( 3 a + b - 2 )$.  Then
\begin{gather}
\label{eq:thm:relations-of-mock-theta-coefficients:third-relation}
  \sum_{\substack{ m \in \ZZ \\ 3 m^2 + m \le 2 n }} \hspace{-0.9em}
  (m + \tfrac{1}{6})\,
  c_h(\omega;\, 2 n - 3 m^2 - m)
\;=\;
  (-1)^{1 + h}
  \sum_{\substack{ a, b \in \ZZ \\ 8 n + 3 = a b}} \hspace{-0.2em}
  \sgn^+(N) \, \sgn^+({\td N}) \;
  \big( |N + \tfrac{1}{3}| - |{\td N} + \tfrac{1}{6}| \big)
\text{,}
\end{gather}
where the sum on the right hand side runs over $a, b$ for which $N, {\td N} \in \ZZ$.

Fix $h \in \{0, 1\}$ and $n \in \frac{1}{2}\ZZ$.  For $a,b\in\ZZ$ with $2 n = a b$ set $N := \frac{1}{6}(a - 3 b - 2)$ and ${\td N} := \frac{1}{6}(a + 3 b - 2)$.  Define
\begin{gather*}
  R_{n}
:=
  \begin{cases}
  \tfrac{2}{3} ( \sigma(\tfrac{n}{2}) - \sigma(n) )
  \text{,}
  &
  \text{if $n \in \ZZ$;}
  \\
  \tfrac{1}{3} ( \sigma(2 n) - 2 \sigma(n) )
  \text{,}
  &
  \text{if $n \in \ZZ + \frac{1}{2}$.}
  \end{cases}
\end{gather*}
Then
\begin{multline}
\label{eq:thm:relations-of-mock-theta-coefficients:fourth-relation}
  \sum_{\substack{ m \in \ZZ \\ 3 m^2 + 2 m + 1 \le 2 n }} \hspace{-1.2em}
  (m + \tfrac{1}{3})\,
  c_h(\omega;\, 2 n - 3 m^2 - 2 m - 1)
\\
=
  (-1)^{h} R_{n}
  +
  (-1)^{1 + h}
  \sum_{\substack{a, b \in \ZZ\\ 2 n = a b}}
  \sgn^+(N)\, \sgn^+({\td N}) \;
  \big( |N + \tfrac{1}{3}| - |{\td N} + \tfrac{1}{3}| \big)
\text{,}
\end{multline}
where the sum on the right hand side runs over $a, b$ for which ${\td N} \in \ZZ$, and $N \in 2 \ZZ + 1$ if $h = 0$ and $N \in 2 \ZZ$ if $h = 1$.
\end{theorem}
\begin{proof} 
The proof is completely analogous to the proof of Theorem~\ref{thm:relations-of-mock-theta-coefficients}.  Write $F^+=\rT(F_0^+,F_1^+,F_2^+)$ and again $G=\rT(G_0,G_1,G_2)$.  It is easy to extend Lemma~\ref{la:mock-theta-holomorphic-projection-coefficients} in each of the cases below.
\begin{enumerate}[(a)]

\item
Relation~\eqref{eq:thm:relations-of-mock-theta-coefficients:second-relation} follows from considering the second and third component of \eqref{eq:restate-of-hol(FG)}, and more precisely from considering:  $F_0^+ (G_1 - G_2)$ (for $n \in \ZZ$) and $F_0^+ (- G_1 - G_2)$ (for $n \in \ZZ + \frac{1}{2}$).

\item
Relation~\eqref{eq:thm:relations-of-mock-theta-coefficients:third-relation} follows from considering the fourth and seventh component of \eqref{eq:restate-of-hol(FG)}, and more precisely from considering: $(F_1^+ + F_2^+) G_0$ (for $h = 0$) and $(F_1^+ - F_2^+) G_0$ (for $h = 1$).

\item
Relation~\eqref{eq:thm:relations-of-mock-theta-coefficients:fourth-relation} follows from considering the fifth, sixth, eighth, and ninth component of \eqref{eq:restate-of-hol(FG)}, and more precisely from considering: $(F_1^+ + F_2^+) (G_1 - G_2)$ (for $h = 0$, $n \in \ZZ + \frac{1}{2}$), $(F_1^+ - F_2^+) (G_1 - G_2)$ (for $h = 1$, $n \in \ZZ$), $(F_1^+ + F_2^+) (- G_1 - G_2)$ (for $h = 0$, $n \in \ZZ$), and $(F_1^+ - F_2^+) (-G_1 - G_2)$ (for $h = 1$, $n \in \ZZ + \frac{1}{2}$).
\qedhere
\end{enumerate}
\end{proof}

\bibliographystyle{amsalpha}
\bibliography{Lit}

\providecommand{\bysame}{\leavevmode\hbox to3em{\hrulefill}\thinspace}
\providecommand{\MR}{\relax\ifhmode\unskip\space\fi MR }
\providecommand{\MRhref}[2]{%
  \href{http://www.ams.org/mathscinet-getitem?mr=#1}{#2}
}
\providecommand{\href}[2]{#2}
\begin{thebibliography}{GMO13}

\bibitem[ARZ13]{And-Rho-Zwe-ANT13}
G.~Andrews, R.~Rhoades, and S.~Zwegers, \emph{Modularity of the concave
  composition generating function}, Algebra {N}umber {T}heory \textbf{{\bf 7}}
  (2013), no.~9, 2103--2139.

\bibitem[BF04]{B-F-Duke04}
J.~Bruinier and J.~Funke, \emph{On two geometric theta lifts}, Duke {M}ath.\
  {J}. \textbf{{\bf 125}} (2004), no.~1, 45--90.

\bibitem[BO06]{B-O-Invent06}
K.~Bringmann and K.~Ono, \emph{The $f(q)$ mock theta function conjecture and
  partition ranks}, Invent.\ {M}ath. \textbf{{\bf 165}} (2006), no.~2,
  243--266.

\bibitem[Bor99]{Borcherds-Gross-Kohnen-Zagier}
R.~Borcherds, \emph{The {G}ross-{K}ohnen-{Z}agier theorem in higher
  dimensions}, Duke {M}ath.\ {J}. \textbf{{\bf 97}} (1999), no.~2, 219--233.

\bibitem[Bru02]{B-LNM}
J.~Bruinier, \emph{Borcherds products on {O}(2, {$l$}) and {C}hern classes of
  {H}eegner divisors}, \rm Lecture {N}otes in {M}ath., vol.~{\bf 1780},
  Springer, 2002.

\bibitem[BYZ04]{B-Y-Z-Crelle04}
B.~Berndt, A.~Yee, and A.~Zaharescu, \emph{New theorems on the parity of
  partition functions}, J.\ {R}eine\ {A}ngew.\ {M}ath. \textbf{{\bf 566}}
  (2004), 91--109.

\bibitem[CDH12]{Ch-Du-Ha}
M.~Cheng, J.~Duncan, and J.~Harvey, \emph{{U}mbral moonshine}, Preprint, 2012.

\bibitem[{\relax DLMF}]{NIST}
\emph{{{NIST} {D}igital {L}ibrary of {M}athematical {F}unctions}},
  http://dlmf.nist.gov/, Release 1.0.6 of 2013-05-06.

\bibitem[DMZ12]{DMZ}
A.~Dabholkar, S.~Murthy, and D.~Zagier, \emph{Quantum black holes and mock
  modular forms}, Preprint, 2012.

\bibitem[EOT11]{Egu-Oog-Tac}
T.~Eguchi, H.~Ooguri, and Y.~Tachikawa, \emph{Notes on the ${K}3$ surface and
  the {M}athieu group ${M}24$}, Exp.\ Math. \textbf{{\bf 20}} (2011), no.~1,
  91--96.

\bibitem[GMO13]{Ono-Gri-Mal-Forum}
M.~Griffin, A.~Malmendier, and K.~Ono, \emph{${SU}(2)$ {D}onaldson invariants
  for the projective plane}, To appear in Forum Mathematicum, 2013.

\bibitem[GR07]{GR-tables}
I.S. Gradshteyn and I.M. Ryzhik, \emph{Table of integrals, series, and
  products}, seventh ed., Elsevier/{A}cademic {P}ress, {A}msterdam, 2007.

\bibitem[GZ86]{G-Z-Invent86}
B.~Gross and D.~Zagier, \emph{Heegner points and derivatives of {$L$}-series},
  Invent.\ {M}ath. \textbf{{\bf 84}} (1986), no.~2, 225--320.

\bibitem[Hur85]{Hur-MathAnn1885}
A.~Hurwitz, \emph{Ueber {R}elationen zwischen {C}lassenanzahlen bin\"arer
  quadratischer {F}ormen von negativer {D}eterminante}, Math.\ {A}nn.
  \textbf{{\bf 25}} (1885), no.~2, 157--196.

\bibitem[HZ76]{Hi-Za-Invent76}
F.~Hirzebruch and D.~Zagier, \emph{Intersection numbers of curves on {H}ilbert
  modular surfaces and modular forms of {N}ebentypus}, Invent.\ {M}ath.
  \textbf{{\bf 36}} (1976), 57--113.

\bibitem[Iwa02]{Iw-2002}
H.~Iwaniec, \emph{Spectral methods of automorphic forms}, \rm {G}raduate
  {S}tudies in {M}athematics, vol.~{\bf 53}, {AMS}, 2002.

\bibitem[KZ95]{Kan-Z}
M.~Kaneko and D.~Zagier, \emph{{\it A generalized {J}acobi theta function and
  quasimodular forms}, {\rm in: {t}he moduli space of curves ({t}exel {i}sland,
  1994)}}, \rm {P}rogr. {M}ath. {\bf 129}, pp.~165--172, Birkh\"auser, 1995.

\bibitem[MO12a]{Ono-Mal-CNTP12}
A.~Malmendier and K.~Ono, \emph{Moonshine for ${M}_{24}$ and {D}onaldson
  invariants of $\mathbb{CP}^2$}, Commun.\ {N}umber {T}heory {P}hys.
  \textbf{{\bf 6}} (2012), no.~4, 759--770.

\bibitem[MO12b]{Ono-Mal-GeomTop12}
\bysame, \emph{${SO}(3)$-{D}onaldson invariants of $\mathbb{CP}^2$ and mock
  theta functions}, Geom.\ {T}opol. \textbf{{\bf 16}} (2012), no.~3,
  1767--1833.

\bibitem[NRS98]{Ni-Ru-Sa-JNT98}
J.~Nicolas, I.~Ruzsa, and A.~S{\'a}rk{\"o}zy, \emph{On the parity of additive
  representation functions}, J.\ Number Theory \textbf{{\bf 73}} (1998), no.~2,
  292--317.

\bibitem[Ono09]{Ono-CDM08}
K.~Ono, \emph{Unearthing the visions of a master: harmonic {M}aass forms and
  number theory}, Current developments in mathematics, 2008, Int.\ {P}ress,
  {S}omerville, {MA}, 2009, pp.~347--454.

\bibitem[Sko08]{Sko-Weil}
N-P. Skoruppa, \emph{Jacobi forms of critical weight and {W}eil
  representations}, Modular forms on {S}chiermonnikoog, Cambridge {Univ}.\
  Press, Cambridge, 2008, pp.~239--266.

\bibitem[Stu80]{Sturm-holproj}
J.~Sturm, \emph{Projections of {$C^{\infty }$} automorphic forms}, Bull. Amer.
  Math. Soc. (N.S.) \textbf{{\bf 2}} (1980), no.~3, 435--439.

\bibitem[Zag09]{Z-Bourbaki}
D.~Zagier, \emph{Ramanujan's mock theta functions and their applications
  [d'apr\`es {Z}wegers and {B}ringmann-{O}no]}, Ast\'erisque (2009), no.~326,
  Exp.\ {N}o.\ 986, vii--viii, 143--164 (2010), S{\'e}minaire {B}ourbaki.
  {V}ol. 2007/2008.

\bibitem[Zwe01]{Zwe-Contemp01}
S.~Zwegers, \emph{Mock $\vartheta$-functions and real analytic modular forms},
  {$q$}-series with applications to combinatorics, number theory, and physics
  ({U}rbana, {IL}, 2000), Contemp.\ {M}ath., vol.~{\bf 291}, Amer.\ {M}ath.\
  {S}oc., Providence, {RI}, 2001, pp.~269--277.

\bibitem[Zwe02]{Zwe-thesis}
\bysame, \emph{Mock theta functions}, Ph.D. thesis, Universiteit Utrecht, The
  Netherlands, 2002.

\end{thebibliography}

\end{document}